\documentclass[a4paper,11pt]{amsart}
\usepackage{latexsym,amssymb,amsmath,amsopn,graphics,xy,epsfig,picture,epic}
\usepackage{color}

\renewcommand{\le}{\leqslant}
\renewcommand{\leq}{\leqslant}
\renewcommand{\ge}{\geqslant}
\renewcommand{\geq}{\geqslant}

\DeclareMathOperator{\supp}{supp}

\newcommand{\nn}{\mathcal{N}}

\newcommand{\R}{\mathbb{R}}
\renewcommand{\P}{\mathbb{P}}
\newcommand{\N}{\mathbb{N}}
\newcommand{\C}{\mathbb{C}}

\newcommand{\ab}{{(\alpha,\beta)}}

\newcommand{\M}{\mathbb{M}}

\newcommand{\Hb}{\mathbb{H}}
\renewcommand{\S}{\mathbb{S}}

\newcommand{\ud}{\mathrm{d}}

\newcommand{\rp}{\right)} 
\newcommand{\lp}{\left(}

\usepackage[bookmarks,bookmarksopen=false,
                  pdfauthor={Autor},%
                  pdftitle={Titel},%
                  colorlinks=true,%
                  linkcolor=blue,%
                  citecolor=blue,%
                  urlcolor=blue,%
                ]{hyperref}

\newtheorem*{theorem*}{Theorem}

\newtheorem{lemma}{Lemma}[section]

\newtheorem{theorem}[lemma]{Theorem}

\theoremstyle{definition}

\newtheorem{remark}[lemma]{Remark}

\begin{document}



\title[Concentration of spherical harmonics]{Concentration estimates for finite expansions of  spherical harmonics on two-point homogeneous spaces\\ via the large sieve principle}

\author{Philippe Jaming}
\email{philippe.jaming u-bordeaux.fr}

\author{Michael Speckbacher}
\email{speckbacher@kfs.oeaw.ac.at}

\address{Univ. Bordeaux, IMB, UMR 5251, F-33400 Talence, France. CNRS, IMB, UMR 5251, F-33400 Talence, France.}

\date{} 

\begin{abstract}
\noindent We study the concentration problem on compact two-point homogeneous spaces  
of finite expansions of eigenfunctions of the Laplace-Beltrami operator using large sieve
methods. We derive upper bounds for concentration in terms of the
maximum Nyquist density. Our proof uses estimates of the spherical harmonics basis coefficients of certain zonal filters and an ordering result for Jacobi polynomials for arguments close to one. 
\end{abstract}


\maketitle

\section{Introduction}

\noindent The large sieve principle is a family of inequalities for trigonometric polynomials which has become a standard tool in analytic number theory, \emph{see e.g.} \cite{mont78}. In the same article, Montgomery mentions on p. 562 that Bombieri  (in an unpublished work) derived the following inequality with similar arguments.

Let $t\in [0,1]$ and 
$\displaystyle f(t)=\sum_{k=1}^K a_k e^{2\pi i kt}$.
If $\mu$ is a positive measure and $0<\delta<1$, then
\begin{equation}\label{eq:carleson}
\int_0^1 |f(t)|^2\,\ud \mu(t)\leq (K+2\delta^{-1})\cdot\sup_{t\in[0,1]}\mu([t,t+\delta])\cdot\int_{0}^1 |f(t)|^2\,\ud t.
\end{equation}
In particular, if $\mu$ is given by $\chi_\Omega(t)\,\ud t$, for $\Omega\subset [0,1]$ measurable, then
\begin{equation}\label{eq:ls-bombi}
\int_\Omega |f(t)|^2\,\ud t\leq 3\cdot\rho(\Omega,K)\cdot\int_{0}^1 |f(t)|^2\,\ud t,
\end{equation}
where the so called \emph{maximum Nyquist density} is given by
\begin{equation}\label{eq:circl-max-Nyq}
\rho(\Omega,K)=K\cdot\sup_{t\in[0,1]}\big|\Omega\cap [t,t+1/K]\big|.
\end{equation}
Donoho and Logan \cite{dolo92} first observed that \eqref{eq:ls-bombi} gives a  particularly strong concentration estimate if the set $\Omega$ is ``\emph{sparse}''. If only a small portion of $\Omega$ is contained in any interval of length $1/K$, then it follows that the energy of $f$ cannot be well concentrated on $\Omega$.  This lead them to derive   similar inequalities for functions in the $L^p$-Paley-Wiener spaces on the real line, $p\in\{1,2\}$. Recently, this idea has been adapted  for concentration problems of time-frequency distributions \cite{abspe17-sampta,abspe18-sieve} and finite spherical harmonics expansions on the 2-sphere \cite{spehry19}.

A common approach to study concentration problems was introduced by Landau, Slepian and Pollak \cite{lan61,lapo62,slepo61} in a series of papers nowadays known as the ``\emph{Bell-Lab papers}''. This approach has been frequently adapted and applied to various function spaces. In the context of this paper, we  refer to  \cite{alsasne99,sidawi06} for a treatment of the Bell-Lab theory of finite expansions of spherical harmonics on the 2-sphere. We would also like to mention that measures that allow for an inequality like in  \eqref{eq:carleson} have been studied in a multitude of contributions and are commonly referred to as \emph{Carleson measures}. For a result concerning Carleson measures on compact manifolds, \emph{see} \cite{orpri13}.

It is the main purpose of this article to generalize the results in \cite{spehry19} from the 2-sphere to general two-point homogeneous spaces.  A compact Riemannian manifold $\M$ with metric $\ud(\cdot,\cdot)$ is called \emph{two-point homogeneous} if for every four points $x_1,x_2,y_1,y_2\in\M$ satisfying $\ud(x_1,x_2)=\ud(y_1,y_2)$, there exists an isometry $I:\M\rightarrow\M$ such that $I(x_i)=y_i,\ i=1,2$. These  spaces were fully characterized by Wang \cite{wa52}, see also \cite{ca29,ga67,hel62,hel65,ko73}. We recall the full list of these spaces in Section~\ref{sec:class-2p}
but note that they include the sphere in $\R^d$ as well as the real projective spaces.

Let $H_k$ be the $k$\emph{-th} eigenspace of the  Laplace-Beltrami operator on $\M$ associated to the eigenvalue $\lambda_k$, $k\in \mathcal{N}$ ($\mathcal{N}=\N_0$ or $2\N_0$
depending on $\M$). Then $L^2(\M)=\bigoplus_{k\in\mathcal{N}}H_k$, where $L^2(\M)$ is equipped with the  invariant Haar measure $\nu$. We define the space of \emph{finite spherical harmonics expansions} by 
$$
\mathcal{S}_K=\bigoplus_{k\in\mathcal{N},k\leq K}H_k.
$$
Let  $\Omega\subset\M$ be measurable, and $1\leq p<\infty$. 
We study the concentration problem 
\begin{equation}\label{eq:lambda-p}
\int_\Omega |f(x)|^p\,\ud\nu(x)\leq \lambda_p(\Omega,K)\cdot\int_\M |f(x)|^p\,\ud\nu(x),\quad f\in \mathcal{S}_K,
\end{equation}
and seek for simple estimates of   $\lambda_p(\Omega,K)$ in terms of a  maximum Nyquist density adapted to two-point homogeneous spaces.
Here, the interval $[t,t+1/K]$ in \eqref{eq:circl-max-Nyq} is replaced by the \emph{geodesic cap}
\footnote{We could as well consider geodesic balls $B(y,r)=\{x\in\M:\ \mathrm{d}(x,y)<r\}$.
Note that $\mathcal{C}_\delta(y)=B(y,\gamma^{-1}\arccos\delta)$ or equivalently 
$B(x,r)=\mathcal{C}_{\cos\gamma r}(x)$. It turns out that caps are more convenient here.}
centered at $y$ which is given by
$$
\mathcal{C}_\delta(y):=\big\{x\in\M:\ \cos\big(\gamma\,\ud(x,y)\big)\geq\delta\big\},
$$
where $\gamma$  (given in Section~\ref{sec:class-2p}) depends on the length of the closed geodesics only. The \emph{maximum Nyquist density} is defined as
\begin{equation}\label{def-max-nyqu}
\rho(\Omega,K):=\sup_{y\in\M} \frac{|\Omega\cap \mathcal{C}_{t_K(\M)}(y)| }{|\mathcal{C}_{t_K(\M)}(y)|},
\end{equation}
where $t_K(\M)$ depends on $K$ and $\M$ only and is
the largest zero of a Jacobi polynomial $P_K^\ab$, where the values of $\alpha$ and $\beta$  
depend on $\M$ only. The explicit expression can be found in \eqref{def-alpha-beta}.

Our main contribution is the following estimate in terms of the maximum Nyquist density. See Theorem~\ref{thm:l2} for a more detailed account of the result.

\begin{theorem*}
Let $\M$ be a two-point homogeneous space,  $\mu$ be a $\sigma$-finite measure, $\Omega\subset\M$ be
measurable, and $t_K(\M)\leq\delta<1$. For $K \in\mathcal{N}$ and every
$f\in \mathcal{S}_K$, it holds
\begin{equation*} 
\int_{\M}|f(x)|^2\,\ud\mu(x)\leq D_{K,\delta}\cdot
\sup_{y\in\M} \mu(\mathcal{C}_\delta(y))\cdot
 \int_{\M}|f(x)|^2\,\ud\nu(x),
\end{equation*}
and
\begin{equation}
\label{eq:conintro}
\int_\Omega |f(x)|^2\,\ud\nu(x) \leq A_{K}\cdot \rho(\Omega,K)\int_\M |f(x)|^2\,\ud\nu(x),
\end{equation}
where $A_K$ and $D_{K,\delta} $ are given explicitly in \eqref{eq-c2l} and \eqref{nonuniform-L2-mu} respectively.
\end{theorem*}

We will also show that $A_K$ converges to a constant given in Lemma~\ref{lemma-binfty} when $K\to\infty$.

In the case of the 2-sphere,  we exactly recover  the results from \cite{spehry19}.
We also derive $L^p$-estimates for $1<p<\infty$ via interpolation and duality arguments in Theorem~\ref{thm:Lp-est}.

\smallskip

We conclude this introduction by mentioning the fact that concentration properties of eigenfunctions
of the Laplacian on Riemannian manifolds have been extensively studied ({\it see e.g.} \cite{Zelditch} and references therein). Our result here is of a slightly different nature to most results so far. Our main contribution is when the set $\Omega$ is sparse
in the sense that $\rho(\Omega,K)$ is much smaller than the measure of $\Omega$.
Note also that our result \eqref{eq:conintro} applies to functions $f$ that are 
linear combinations of eigenfunctions of the 
Laplace-Beltrami operator for different eigenvalues rather than to the more common situation of a single eigenfunction. On the other hand, when $f$ is a single eigenfunction of the Laplace-Beltrami operator,
there seems to be no improvement in the concentration bound \eqref{eq:conintro} obtained through our method of proof.

\smallskip

The remaining of this paper is organized as follows: we start with a section of preliminaries on Jacobi polynomials
and the incomplete Beta function. Section~\ref{sec:two-point} is then devoted to two point homogeneous manifolds and their spherical harmonics. Once this is done, we can conclude with the proof of the main theorem in Section~\ref{sec:LSE}.

\section{Preliminaries}
\subsection{Jacobi Polynomials}

For $\alpha,\beta>-1$, consider the Jacobi weight
$$
\omega_{\alpha,\beta}(t)=(1-t)^\alpha(1+t)^\beta.
$$
The \emph{Jacobi polynomials} $P_n^\ab$, $n\in\N_0=\{0,1,\ldots\}$, are then a family of orthogonal
polynomials in $L^2_{\omega_{\alpha,\beta}}(-1,1)$ satisfying the
orthogonality relations 
\begin{multline}\label{eq:ortho-jacobi}
\int_{-1}^1P_n^{(\alpha,\beta)}(t)P_m^{(\alpha,\beta)}(t)\omega_{\alpha,\beta}(t)\,\mathrm{d}t\\
=\frac{2^{\alpha+\beta+1}\Gamma(n+\alpha+1)\Gamma(n+\beta+1)}{n!(2n+\alpha+\beta+1)\Gamma(n+\alpha+\beta+1)}\delta_{n,m}.
\end{multline} 
where $n,m\in\N_0$, and $\delta_{n,m}$ denotes the Kronecker delta.
Note for future use that, by parity,
\begin{multline}\label{eq:ortho-jacobi01}
\int_{0}^1P_{2n}^{(\alpha,\alpha)}(t)P_{2m}^{(\alpha,\alpha)}(t) 
\omega_{\alpha,\alpha}(t)\,\mathrm{d}t\\
=\frac{4^{\alpha}\Gamma(2n+\alpha+1)^2}{(2n)!(4n+2\alpha +1)\Gamma(2n+2\alpha+1)}\delta_{n,m}.
\end{multline}
The Jacobi polynomials are explicitly given by
\begin{equation*}
P_n^\ab(t):=\frac{\Gamma(n+\alpha+1)}{n!\Gamma(n+\alpha+\beta+1)}\sum_{m=0}^n\binom{n}{m}\frac{\Gamma(n+m+\alpha+\beta+1)}{\Gamma(m+\alpha+1)}\left(\frac{t-1}{2}\right)^m
\end{equation*}
or by the Rodrigues Formula
\begin{equation*}
P_n^\ab(t)
= \frac{(-1)^n}{2^n n!} (1-t)^{-\alpha} (1+t)^{-\beta}
\frac{\ud^n}{\ud z^n} \left\{ (1-t)^\alpha (1+t)^\beta (1 - t^2)^n \right\}~.
\end{equation*}
Jacobi polynomials satisfy the following symmetry relation
\begin{equation}\label{eq:jacobi-symmetry}
P_n^\ab(-t)=(-1)^nP_n^{(\beta,\alpha)}(t).
\end{equation}
Throughout this paper we are only concerned with $\alpha\geq-\frac{1}{2}$ and $\alpha\geq\beta>-1$. In that case, 
 one has   \cite[18.14.1]{NIST10}
\begin{equation}\label{pn-max} 
 |P_n^\ab(t)|\leq P_n^\ab(1)=\dfrac{\Gamma(n+\alpha+1)}{n!\Gamma(\alpha+1)}.
 \end{equation} 
By \cite[18.9.15]{NIST10}, the derivative of $P_n^\ab$ satisfies
\begin{equation}\label{pn-prime-max}
 \frac{\mathrm{d}}{\mathrm{d}t}P_n^\ab(t)=\frac{1}{2}(n+\alpha+\beta+1)P^{(\alpha+1,\beta+1)}_{n-1}(t),
\end{equation}
which implies that  $\left|\displaystyle\frac{\mathrm{d}}{\mathrm{d}t}P_n^\ab(t)\right|\lesssim n^{\alpha+2}$.

Like all orthogonal polynomials, Jacobi polynomials satisfy a three term
recurrence relation, \emph{see e.g.} \cite[18.2.8]{NIST10},
\begin{equation}\label{recursion-ordinary}
P_{n+1}^\ab(t)=(A_n t+B_n)P_n^\ab(t)-C_nP_{n-1}^\ab(t), \qquad n = 1,2,\ldots,
\end{equation}
with $P_0^\ab(t)=1$, and $P_1^\ab(t)=\frac{1}{2}\big((\alpha+\beta+2)t+\alpha-\beta\big)$, and $A_n,B_n,C_n\in \R$. Although we do not need there explicit expressions in this paper, note that they are given by
$A_n=a_n/d_n$, $B_n=b_n/d_n$, $C_n=c_n/d_n$ with
\begin{eqnarray*}
a_n&=&(2n+\alpha+\beta)(2n+\alpha + \beta+1) (2n+\alpha + \beta+2)\\
b_n&=&(2n+\alpha + \beta+1) (\alpha^2 - \beta^2)\\
c_n&=&(n+\alpha ) (n + \beta) (2n+\alpha + \beta+2)\\
d_n&=&2(n+1) (n + \alpha + \beta+1) (2n + \alpha + \beta).
\end{eqnarray*}
We will however need that $a_n,c_n,d_n>0$ and that
$
P_n^\ab(1)> 0.
$

It follows from general theory of orthogonal polynomials
that all zeros of $P_n^\ab$ lie in the interval $(-1,1)$, \emph{see e.g.}
\cite[18.2(vi)]{NIST10}. Further, it is known \cite[18.2(vi)]{NIST10} that the zeroes of $P_n^\ab$
and $P_{n+1}^\ab$ interlace. That is, if we  write $t_{j,k}$ for the $k$\emph{-th} zero of $P_n^\ab$,
$$
t_{n+1,1}<t_{n,1}<t_{n+1,2}<\cdots<t_{n+1,n-1}<t_{n,n}<t_{n+1,n+1}.
$$

\noindent Let us write $\theta_{n,1}:=\arccos(t_{n,n})$. If either $\alpha,\beta\in \left[-\frac{1}{2},\frac{1}{2}\right]$, or $\alpha+\beta\geq -1,$ and $\alpha>-1/2$, then  the following asymptotic behavior follows from \cite[18.16.6, 18.16.7, and 18.16.8]{NIST10}
\begin{equation}\label{eq:asymp-tKK}
\theta_{n,1}=\frac{j_{\alpha,1}}{n}+\mathcal{O}\left(n^{-2}\right),
\end{equation}
where $j_{\alpha,m}$ denotes the $m$\emph{-th} positive zero of the Bessel function of the first kind $J_\alpha$. Note that throughout this paper we will only have to deal with situations where $\alpha,\beta\geq -1/2$. Consequently, at least one of the two conditions above will always be satisfied.
Taking the cosine
of both sides yields
\begin{equation}\label{zeros-appr1}
t_{n,n}=1-\frac{j_{\alpha,1}^2}{2n^2}+\mathcal{O}\left(n^{-3}\right).
\end{equation}

\noindent The \emph{Mehler-Heine formula} \cite[18.11.5]{NIST10} describes the
asymptotic behavior of $P_n^\ab$ at arguments approaching~$1$
\begin{equation}\label{mehler-heine1}
\lim_{n\rightarrow\infty}\frac{1}{n^\alpha} P_n^\ab \lp1-\frac{z^2}{2n^2}\rp =
\frac{2^\alpha}{z^\alpha}J_\alpha(z).
\end{equation}
Precise lower bounds for $t_{n,n}$ have been obtained recently in \cite{Nikolov1,Nikolov2}. We are rather interested in an upper bound which can be derived from the Euler-Rayleigh technique.
A simple computation derived from those in \cite{Nikolov1,Nikolov2} shows that
\begin{equation}
\label{eq:niko}
t_{n,n}\leq 1-2\frac{\alpha+1}{n(n+\alpha+\beta+1)}.
\end{equation}
We conclude this section with the following lemma that describes a certain monotonicity
property of Jacobi polynomials. It was already shown for the special case of  Legendre polynomials in \cite[Lemma 2.1]{spehry19}.

\begin{lemma}\label{ordering-PN}
Let $n\ge 1$ be fixed and $t\in[t_{n,n},1)$. For $k=1,\ldots,n$, one has
\begin{equation}\label{order-i}
\frac{P_k^\ab(t)}{P_k^\ab(1)}<\frac{ P_{k-1}^\ab(t)}{P_{k-1}^\ab(1)},
\end{equation}
consequently $P_k^\ab(t) \ge 0$.
\end{lemma} 

\begin{proof} 
We show \eqref{order-i} by induction with respect to~$n$. For
$n=1$, we only have to consider $k = 1$. But
$P_1^\ab(t) = \frac{1}{2}\big((\alpha+\beta+2)t+\alpha-\beta\big)$ so that 
$t_{1,1} = \frac{\beta-\alpha}{\alpha+\beta+2}$ while $P_0^\ab(t) = 1$. 
As $\alpha+\beta+2>0$, it follows that $P_1^\ab$ is an increasing linear function and thus
$0\leq \frac{P_1^\ab(t)}{P_1^\ab(1)}<1$ on $[t_{1,1},1)$ and
\eqref{order-i} is true. 

We now assume that \eqref{order-i} holds
for $k \leq n$. It follows from \eqref{recursion-ordinary} that there exist $\widetilde A_n, \widetilde B_n,$ and $\widetilde C_n$ so that 
$$
\frac{P_{n+1}^\ab(t)}{P_{n+1}^\ab(1)}=(\widetilde A_n t+\widetilde B_n)\frac{P_n^\ab(t)}{P_{n}^\ab(1)}-\widetilde C_n\frac{P_{n-1}^\ab(t)}{P_{n-1}^\ab(1)}, \qquad n = 1,2,\ldots.
$$
The coefficients are expressed in terms of $a_n,b_n,c_n,d_n$ and the values of Jacobi polynomials at $1$.
We only need to notice that
$$
\widetilde A_n=\frac{a_nP_{n}^\ab(1)}{d_nP_{n+1}^\ab(1)}>0
\quad\mbox{and}\quad
\widetilde C_n=\frac{c_nP_{n-1}^\ab(1)}{d_nP_{n+1}^\ab(1)}>0.
$$
Setting $t=1$ in the above equation, it follows that
$1=\widetilde A_n +\widetilde B_n-\widetilde C_n$. We therefore have for every 
$t \in [t_{n+1,n+1},1)\subset[t_{n,n},1)$
\begin{eqnarray*}
\frac{P_{n+1}^\ab(t)}{P_{n+1}^\ab(1)} 
&=&(\widetilde A_n t+\widetilde B_n) \frac{P_n^\ab(t)}{P_{n}^\ab(1)}
-\widetilde C_n \frac{P^\ab_{n-1}(t)}{P_{n-1}^\ab(1)}\\
&<&(\widetilde A_n +\widetilde B_n)\frac{P_n^\ab(t)}{P_{n}^\ab(1)}
-\widetilde C_n \frac{P^\ab_{n-1}(t)}{P_{n-1}^\ab(1)}
\end{eqnarray*}
since $\widetilde A_n >0$. Using the induction hypothesis and the fact that $\widetilde C_n\geq 0$ we get
\begin{equation*}
\frac{P_{n+1}^\ab(t)}{P_{n+1}^\ab(1)}
< (\widetilde A_n +\widetilde B_n)\frac{P_n^\ab(t)}{P_{n}^\ab(1)}-\widetilde C_n\frac{P_n^\ab(t)}{P_{n}^\ab(1)}
= \frac{P_n^\ab(t)}{P_{n}^\ab(1)}.
\end{equation*}
This implies \eqref{order-i} with $k = n+1$ and the inductive proof is complete.

Since $t_{n,n}$ is the largest zero of $P_n^\ab$ and $P_n^\ab(1) = 1$, it
follows that $P_n^\ab(t) \ge 0$ for $t\in[t_{n,n},1)$. Consequently, for $k=1,\ldots,n$, we have
$$
\frac{P_k^\ab(t)}{ P_k^\ab(1)}\ge \frac{P_n^\ab(t)}{P_n^\ab(1)} \ge 0.
$$
\end{proof}

\subsection{Incomplete Beta Functions}
For $a,b>0$,
the \emph{beta function} is given by
$$
B(a,b):=\int_0^1 t^{a-1}(1-t)^{b-1}\ud t=\frac{\Gamma(a)\Gamma(b)}{\Gamma(a+b)},
$$
 and for $x\in(0,1)$, the \emph{incomplete beta function} is defined as
$$
B_x(a,b):=\int_0^x t^{a-1}(1-t)^{b-1}\ud t.
$$
It satisfies the following relation
\cite[8.17.7]{NIST10}
$$
B_x(a,b)=\frac{x^a}{a}F(a,1-b;a+1;x)
$$
where $F={}_2F_1$ denotes the hypergeometric function.
As the series defining $F$ converges absolutely for arguments with absolute value less than $1$,
it follows that 
\begin{equation*}
B_x(a,b)=\frac{x^a}{a}+\mathcal{O}(x^{1+a}),\qquad \mbox{as } x\rightarrow 0.
\end{equation*}
Let $\delta>0$. The size of a geodesic cap will depend on 
\begin{eqnarray}
\int_\delta^1(1-t)^\alpha(1+t)^\beta \ud t&=&2^{\alpha+\beta+1} B_{(1-\delta)/2}(\alpha+1,\beta+1)\nonumber \\
&=&\frac{2^{\beta}(1-\delta)^{\alpha+1}}{\alpha+1}+\mathcal{O}((1-\delta)^{\alpha+2}).
\label{eq:incomplete-beta-asym}
\end{eqnarray}

\section{Two-Point Homogeneous Spaces}\label{sec:two-point}

\subsection{Classification of Two-Point Homogeneous Spaces}\label{sec:class-2p}

In this paper we consider two-point homogeneous spaces, or, in other terminology,  compact globally symmetric spaces of rank $1$. A compact Riemannian manifold $\M$ with metric $\ud(\cdot,\cdot)$ is called \emph{two-point homogeneous} if for every four points $x_1,x_2,y_1,y_2\in\M$ satisfying $\ud(x_1,x_2)=\ud(y_1,y_2)$, there exists an isometry $I:\M\rightarrow\M$ such that $I(x_i)=y_i,\ i=1,2$. 
A full classification of the two-point homogeneous spaces was given by Wang \cite{wa52}. The complete list of two-point homogeneous spaces is given by 

\medskip
\begin{tabular}{llll}
(i) &the $d$-dimensional sphere & $\S^d$, & $d=1,2,3,\ldots$,\\
(ii)& the real projective space &  $\P^d(\R)$,& $d=2,3,4,\ldots$,\\
(iii)& the complex projective space &$\P^d(\C)$, &$d=4,6,8,\ldots$,\\
(iv) &the quaternion projective space & $\P^d(\mathbb{H})$,& $d=8,12,16,\ldots$,\\
   (v) &the Caley projective space &$\P^{16}(\C a)$.&
\end{tabular}

\medskip
\noindent The superscripts denote the dimension of the corresponding spaces over the reals.
For
further reading on this topic, \emph{see e.g.} Cartan \cite{ca29}, Gangolli \cite{ga67}, and Helgason  \cite{hel65,hel62}.

Each space $\M$
can be considered as the orbit space of some compact subgroup $H$ of the orthogonal
group $G$, that is $\M = G/H$. Let $\pi: G \rightarrow G/H$ be the natural mapping and $e$ be the
identity of $G$. The point $\eta := \pi(e)$ is called the
 \emph{north pole} of $\M$. On any such manifold there is an invariant Riemannian
metric $d(\cdot, \cdot)$, and an invariant Haar measure $d\nu$. 

The geometry of these spaces is in many respects similar. For example, all geodesics
in a given one of these spaces are closed and have the same length $2L$. Here $L$ is the
diameter of $\M$, i.e., the maximum distance between any two points. 
A function on $\M$ is called \emph{zonal} if it  only depends on the
distance of its argument from $\eta$. Since the distance of any point of $\M$ from $\eta$ is
at most $ L$, it follows that a zonal function     can be identified with a function 
on $[0, L]$. 

Two-point homogeneous spaces admit
essentially only one invariant second order differential operator, the Laplace-Beltrami operator $\Delta$.
The eigenvalues of $\Delta$ are given by
$$
\lambda_k=-k(k+\alpha+\beta+1),\quad k\in\nn,
$$
where $\nn=2\N_0$ (the even integers) when $\M=\P^d(\R)$ and $\nn=\N_0$ otherwise, and $\alpha,\beta$ are given in \eqref{def-alpha-beta}.
The corresponding eigenspaces $H_k$ are of finite dimension $d_k:=\dim H_k$,  invariant and irreducible under $G$
 and satisfy
$$
{L^2(\M)=\bigoplus_{k\in\nn} H_k.}
$$
Let $\theta$ be the distance of a point from $\eta$ and $(\theta, u)\in [0,L]\times\M^\bot$ be a geodesic polar
coordinate system, where $u$ is an angular parameter. In this coordinate system the
geodesic component $\Delta_\theta$  of the Laplace-Beltrami operator $\Delta$ has the expression
\begin{equation}\label{eq:laplace-theta}
\Delta_\theta=\frac{1}{\sin (\gamma\hspace{0.05cm}\theta/2)^\sigma\sin(\gamma\hspace{0.05cm}\theta)^\rho}
\frac{\ud}{\ud\theta}\left(\sin( \gamma\hspace{0.05cm}\theta/2)^\sigma \sin(\gamma\hspace{0.05cm}\theta)^\rho\frac{\ud}{\ud\theta}\right),
\end{equation}
where the parameters $\sigma$ and $\rho$ depend on $\M$ and can be found in the 
following list   \cite[p. 171]{hel65}\\
\begin{tabular}{lrllll}
(i) & $\S^d:$ & $\sigma=0$,  &$\rho=d-1$,  &$\gamma=\pi/L$,  &$d=1,2,3\ldots,$ \\
(ii) &  $\P^d(\R):$ &  $\sigma=0,$  &$\rho=d-1,$  &$\gamma=\pi/2L,$ & $d=2,3,4\ldots,$ \\
(iii) & $\P^d(\C):$ & $\sigma=d-2,$  &$\rho=1,$  &$\gamma=\pi/L,$ &$d=4,6,8\ldots,$ \\
(iv) &  $\P^d(\Hb):$  & $\sigma=d-4,$ &$\rho=3,$ &$\gamma=\pi/L,$ &$d=8,12,\ldots,$ \\
(v) &  $\P^{16}(\C a):$ & $ \sigma=8,$  & $\rho=7,$  &$\gamma=\pi/L.$ &  
\end{tabular}

\medskip \noindent Next, define
\begin{equation}\label{def-alpha-beta}
\alpha=\frac{d-2}{2},\quad\mbox{ and }\quad\beta=\frac{\rho-1}{2}.
\end{equation}
Note that $\alpha,\beta\geq -1/2$, and in particular
$$
\alpha+\beta=\frac{d+\rho-3}{2}
\geq -1.
$$
Further, after a change of variables $t=\cos(\gamma\cdot\theta)$, \eqref{eq:laplace-theta} can be written as
$$
\Delta_t=\frac{1}{  (1-t)^\alpha(1+t)^\beta}\frac{\ud}{\ud t}\left( (1-t)^{\alpha+1} (1+t)^{\beta+1}\frac{\ud}{\ud t}\right).
$$
This is just the \emph{Jacobi operator}, and its eigenfunctions are the Jacobi polynomials $P_k^\ab$. 
It follows that the functions $\left\{P_k^\ab\big(\cos(\gamma\,\ud(\cdot,\eta))\big)\right\}_{k\in\nn}$ form an orthogonal basis of the space of zonal functions. 
Here we note that the real projective spaces are different  due to the identification of antipodal points on $\S^d$
so that only even polynomials  appear.

The orthogonality of the Jacobi polynomials also implies that the measure $\nu$ factors as follows 
$$
\ud\nu=\ud\nu^\bot\, (1-\cos(\gamma\hspace{0.05cm}\theta))^\alpha (1+\cos(\gamma\hspace{0.05cm}\theta))^\beta \gamma \sin(\gamma\hspace{0.05cm}\theta)\,\ud\theta.
$$ 
Let us make the following definition
$$
I_\M:=\begin{cases}
(-1,1) & \mbox{if}\ \M\neq \P^d(\R),\\\hspace{0.3cm}(0,1) &  \mbox{if}\ \M=\P^d(\R).
\end{cases}
$$
As we have assumed that $\nu$ is normalized, it thus follows  that
\begin{eqnarray*}
1=\int_\M \,\ud\nu &=&
\int_{\M^\bot}\,\ud\nu^\bot \int_{0}^L (1-\cos(\gamma\theta))^\alpha 
(1+\cos(\gamma\theta))^\beta \gamma\sin(\gamma\theta)\,\ud\theta\\ 
&=& \int_{\M^\bot} \,\ud\nu^\bot  \int_{I_\M} (1-t)^\alpha (1+t)^\beta \,\ud t.
\end{eqnarray*}
If $\M\neq \P^d(\R)$, it follows  by \eqref{eq:ortho-jacobi}  that the measure of the non-geodesic part
$\M^\bot$ is equal to 
\begin{equation}\label{eq:measure-non-geodesic}
\nu^\bot(\M^\bot)= \frac{\Gamma(\alpha+\beta+2)}{2^{\alpha+\beta+1}\Gamma(\alpha+1)\Gamma(\beta+1)}
=2^{-(\alpha+\beta+1)}B(\alpha+1,\beta+1)^{-1}.
\end{equation}
Moreover, by \eqref{eq:ortho-jacobi01} we have
\begin{equation}\label{eq:measure-non-geodesic01}
\nu^\bot(\P^d(\R))= \frac{\Gamma(2\alpha+2)}{4^{\alpha}\Gamma(\alpha+1)^2}=4^{-\alpha}B(\alpha+1,\alpha+1)^{-1}.
\end{equation}

\subsection{Addition Formula }\label{sec:addition-formula}

Let $\{Y_k^j\}_{j=1}^{d_k}$ be an orthonormal basis for $H_k$.
 Then the following 
addition formula can be found in \cite{ko73}
\begin{equation}\label{eq:add-form}
\sum_{j=1}^{d_k}Y_k^j(x)\overline{Y_k^j(y)}=D_k(\M){P_k^\ab}\big(\cos (\gamma\,\ud(x,y))\big).
\end{equation}
For the proof of our main result we need explicit expressions of $D_k(\M)$ and $d_k$ which we  calculate as follows:

First, setting $x=y$, the addition formula immediately implies
\begin{eqnarray}
d_k&=&\sum_{j=1}^{d_k}1\ =\ \sum_{j=1}^{d_k}\int_{\M} |Y_k^j(x)|^2\,\ud\nu(x)
\ =\ \int_{\M} \sum_{j=1}^{d_k}|Y_k^j(x)|^2\,\ud\nu(x)\nonumber\\
&=&\int_{\M} D_k(\M)P_k^{(\alpha,\beta)}( 1)\,\ud\nu(x)\ =\ D_k(\M)P_k^{(\alpha,\beta)}( 1)\,\nu(\M) \nonumber \\
&=&D_k(\M)P_k^{(\alpha,\beta)}( 1),\label{eq:dk-1}
\end{eqnarray}
as $\nu$ is normalized. Note that for every $y\in \M$ 
\begin{equation}\label{eq:norm-P_k}
\|P_k^{(\alpha,\beta)}\big(\cos(\gamma\hspace{0.05cm} \ud(\cdot,\eta))\big)\|_{L^2(\M)}
=\sqrt{\nu^\bot(\M^\bot)}\|P^\ab_k\|_{L^2_{\omega_{\alpha,\beta}}(I_\M)}.
\end{equation}
Now, taking the squared absolute value of both sides in \eqref{eq:add-form} and integrating with respect to $d\nu(x)$, yields
\begin{eqnarray*}
\sum_{j=1}^{d_k}|Y_k^j(y)|^2&=&D_k(\M)^2\int_\M |P_k^\ab\big(\cos(\gamma\,\ud(x,y))\big)|^2\,\ud\nu(x)\\
&=&D_k(\M)^2\nu^\bot(\M^\bot)\|P_k^\ab\|_{L^2_{\omega_{\alpha,\beta}}(I_\M)}^2,
\end{eqnarray*}
as the measure $\nu$ is $G/H$-invariant.
If we then integrate this equation with respect to $y$, it follows that
\begin{equation}
d_k=D_k(\M)^2\nu^\bot(\M^\bot)\|P_k^\ab\|_{L^2_{\omega_{\alpha,\beta}}(I_\M)}^2.
\label{eq:dk-2}
\end{equation}
Comparing \eqref{eq:dk-1} and \eqref{eq:dk-2} and using \eqref{pn-max} and \eqref{eq:ortho-jacobi} then yields
\begin{eqnarray*}
D_k(\M)&=&\frac{P_k^\ab(1)}{\nu^\bot(\M^\bot)\|P^\ab_k\|_{L^2_{\omega_{\alpha,\beta}}(I_\M)}^2}\\
&=&\frac{(2k+\alpha+\beta+1)\Gamma(k+\alpha+\beta+1)\Gamma(\beta+1)}{\Gamma(k+\beta+1)\Gamma(\alpha+\beta+2)}.
\end{eqnarray*}
Consequently, we conclude that dimension of the eigenspace corresponding to the eigenvalue $\lambda_k$ is given by
\begin{eqnarray*}
d_k&=&D_k(\M)P_k^{(\alpha,\beta)}(1)\\
&=&\frac{(2k+\alpha+\beta+1)\Gamma(k+\alpha+\beta+1)}{k!\Gamma(\alpha+\beta+2)}
\frac{\Gamma(k+\alpha+1)\Gamma(\beta+1)}{\Gamma(k+\beta+1)\Gamma(\alpha+1)}.
\end{eqnarray*}
The following relation, which we infer from \eqref{eq:dk-1} and \eqref{eq:dk-2} will be useful later
\begin{equation}\label{eq:sqrt-dk}
d_k^{1/2}=\frac{P_k^\ab(1)}{\nu^\bot(\M^\bot)^{1/2}\|P_k^\ab\|_{L^2_{\omega_{\alpha,\beta}}(I_\M)}}.
\end{equation} 

\subsection{Zonal Functions and Convolutions}\label{sec:zonal&convo}

For a given point $y\in \M$, we define the space of \emph{zonal functions} with respect $y$ by
$$
\mathcal{Z}^p(\M,y):=\left\{f\in L^p(\M):\ f(x)=F\big(\cos(\gamma\,\ud(x,y))\big)\right\}.
$$
Hence every zonal function $f$ can always be identified with a function $F$ defined on $I_\M$.

As $ {P_k^\ab}\big(\cos (\gamma\hspace{0.05cm} d(\cdot, y))\big)$ is in the linear span of the basis elements 
$Y_k^j$ for every $y\in\M$ by the addition formula,
it follows that the space of zonal functions in $H_k$ is at least one dimensional. In the following lemma we show that it is exactly one dimensional.

\begin{lemma}\label{lem:zonal-1d}
For  $k\in\N$, $y\in \M$ and
$
\mathcal{Z}^2_k(\M,y) :=  H_k \cap \mathcal{Z}^2(\M,y),
$
one has $$\dim \mathcal{Z}^2_k(\M,y)=1.$$
\end{lemma}

\begin{proof} Equation  \eqref{eq:add-form} implies that $\dim \mathcal{Z}_k^2(\M,y)\geq 1$. Moreover,
$$
\mathcal{Z}^2(\M,y)\supseteq \bigoplus_{k\in\mathcal{N}}\mathcal{Z}_k^2(\M,y)\supseteq \bigoplus_{k\in\mathcal{N}} \mbox{span}\{P_k^\ab(\cos(\gamma\,\ud(\cdot,y)))\}.
$$
Let us assume to the contrary that there exist $f\in \mathcal{Z}_k^2(\M,y)$, such that 
$$
f\perp P_k^\ab(\cos(\gamma\,\ud(\cdot,y))).
$$
Then, as all the spaces $H_k$ are orthogonal, one has also 
$$
f\perp P_n^\ab(\cos(\gamma\,\ud(\cdot,y)))
$$
for every $n\in \mathcal{N}$ and  consequently 
$$
0=\langle f,P_n^\ab\big(\cos(\gamma\,\ud(\cdot,y))\big) \rangle_{L^2(\M)}
=\nu^\bot(\M^\bot)\langle F,P_n^\ab\rangle_{L^2_{\alpha,\beta}{(I_\M)}},\quad n\in \mathcal{N}.
$$
It is obvious that $\{P_n^\ab\}_{n\in\mathcal{N}}$ is an  orthogonal basis for $L^2_{\omega_{\alpha,\beta}}{(I_\M)}$ 
whenever $\M\neq\P^d(\R)$.  If $\M=\P^d(\R)$, then, since $\alpha=\beta$ and 
$P^{(\alpha,\alpha)}_{2k}(-t)=P^{(\alpha,\alpha)}_{2k}(t)$ by \eqref{eq:jacobi-symmetry}, it follows that  
$\{P_n^{(\alpha,\alpha)}\}_{n\in2\N_0}$ is an  orthogonal basis for the subspace of all even functions in 
$L^2_{\omega_{\alpha,\alpha}}{(-1,1)}$ and thus also for $L^2_{\omega_{\alpha,\alpha}}{(0,1)}$. This finally 
shows that $f=0$.
\end{proof}

As the addition formula holds for arbitrary orthonormal bases and since every $H_k$ contains a one dimensional subspace of zonal functions, we may choose the first basis element of each $H_k$ to be 
$$
Y_k^1(x)=\frac{P_k^{(\alpha,\beta)}\big(\cos(\gamma\,\ud(x,\eta))\big)}{\|P_k^{(\alpha,\beta)}\big(\cos(\gamma\,\ud(\cdot,\eta))\big)\|_{2}},
$$
where $\eta\in \M$ denotes the north pole.
Let us denote the basis coefficients of the orthonormal basis $\{\{Y_k^j\}_{j=1}^{d_k}\}_{k\in\nn}$ by
\begin{equation}\label{def-hat0}
\widehat{f}(k,j):=\int_\M f(x)\overline{Y_k^j(x)}\,\ud\mu(x).
\end{equation}
Let $g\in \mathcal{Z}^q(\M,\eta)$ and $G$ its corresponding function on $I_\M$.
For  $f\in L^p(\M)$ and $1+\frac{1}{r}=\frac{1}{p}+\frac{1}{q}$, we define \emph{convolution} by
$$
(f\ast g)(x)=\int_{\M}f(y)G\big(\cos (\gamma\,\ud(x,y))\big)\,\ud\nu(y),\qquad x\in \M.
$$
Then Young's convolution inequality states that
$$
\|f\ast z\|_r\leq \|f\|_p\|g\|_q.
$$
Like in the euclidean case, this notion of convolution admits a convolution theorem, \emph{i.e.} convolution amounts to a multiplication operator in the domain of the basis coefficients of $Y_j^k$.
This result is known in the special case of $\M=\S^d$, \emph{see e.g.} \cite{freeden1996,kennedy2011,narcowich1996}.
\begin{theorem}\label{conv-thm}
Let $f\in L^2(\M)$, and $g\in \mathcal{Z}^2(\M,\eta)$, then for every $k\in \mathcal{N},\ j\in\{j,\ldots,d_k\}$, it holds
$$
\widehat{(f\ast g)}  (k,j)=(d_k )^{-1/2}\cdot \widehat{f}(k,j)\cdot \widehat{g}(k,1).
$$
\end{theorem}

\begin{proof}
For simplicity, we assume that $g=Y_k^1$. Then $\widehat{g}(k,j)=\delta_{j,1}$, and 
by \eqref{eq:add-form}, \eqref{eq:norm-P_k} and \eqref{eq:dk-2} we  obtain
\begin{eqnarray*}
(\widehat{f\ast g})(k,j)&=&\int_{\M}(f\ast g)(x)\overline{Y_k^j(x)}\,\ud\nu(x)\\
&=&\int_{\M}\int_{\M}f(y)\frac{P_k^{(\alpha,\beta)}\big(\cos(\gamma\,\ud(x,y))\big)}{\big\|P_k^{(\alpha,\beta)}
\big(\cos(\gamma\,\ud(\cdot,\eta))\big)\big\|_2}\overline{Y_k^j(x)}\,\ud\nu(x)\,\ud\nu(y)\\
&=&\frac{1}{D_k(\M)\sqrt{\nu^\bot(\M^\bot)}\|P_k^{(\alpha,\beta)}\|_{L^2_{\omega_{\alpha,\beta}}(I_\M)} }\\
&&\times\int_{\M}\int_{\M}f(y)\sum_{n=1}^{d_k}Y_k^n(x)\overline{Y_k^n(y)}\ \overline{Y_k^j(x)}\,\ud\nu(x)\,\ud\nu(y)\\
&=& (d_k)^{-1/2}\int_{\M}f(y) \overline{Y_k^j(y)}\,\ud\nu(y)=(d_k )^{-1/2}\cdot 
\widehat{f}(k,j)\cdot \widehat{g}(k,1).
\end{eqnarray*}
The general result then follows once we recall that $\dim \mathcal{Z}_\eta^k(\M)=1$, for  $k\in \mathcal{N}$, by 
Lemma~\ref{lem:zonal-1d} and a density argument.
\end{proof}

\section{Large Sieve Estimates}\label{sec:LSE}

\noindent Let $ \delta \in I_\M$ and $y\in \M$. The \emph{geodesic cap} centered in $y$ is defined by
$$
\mathcal{C}_{\delta}(y):=\{x\in \M:\ \cos(\gamma\hspace{0.05cm} d(x,y))\geq  \delta\}.
$$
It is easy to see, after a change of variables, that the measure of $\mathcal{C}_{\delta}(y)$ is  independent of $y\in \M$ and given by
\begin{equation}\label{surf-meas-cap}
|\mathcal{C}_{\delta}(y)|=\nu(\mathcal{C}_{\delta}(y))=2^{\alpha+\beta+1}\nu^\bot(\M^\bot)B_{(1-\delta)/2}(\alpha+1,\beta+1).
\end{equation}
Let $K\in\mathcal{N}$. We denote the space of $K$-finite functions, that is the set of
finite expansions of eigenfunctions of the Laplace-Beltrami operator by
$$
\mathcal{S}_K(\M):=\bigoplus_{k\in\mathcal{N},\ k\leq K} H_k.
$$
In the following, we will give concentration estimates for functions from $\mathcal{S}_K(\M)$. The main idea is 
to construct a certain optimal zonal function that is supported on a geodesic cap, \emph{i.e.} that belongs to
$$
\mathcal{Z}^p_\delta(\M,\eta):=\left\{g\in\mathcal{Z}^p(\M,\eta):\ \mbox{supp}(g)\subset
\mathcal{C}_{\delta}(\eta)\right\}.
$$
The following lemma can be shown exactly as in \cite[Lemma~3.1]{spehry19}. For the convenience of our readers we reproduce the proof here.

\begin{lemma}\label{measure-prop}
Let $\mu$ be a positive $\sigma$-finite measure on $\M$, and let $1 < p,q <
\infty$ be conjugate exponents. If
$g\in\mathcal{Z}^q_\delta(\M,\eta)\setminus\{0\}$, then
\begin{equation}\label{sphere-measure}
\int_{\M}|f|^p\,\ud\mu \leq \sup_{h\in \mathcal{S}_K(\M)} 
\frac{\|h\|_p^p\|g\|_q^p}{\|h\ast g\|_p^p}\cdot \|f\|_p^p\cdot \sup_{y\in \M} \mu(\mathcal{C}_{\delta}(y)),
\qquad f\in \mathcal{S}_K(\M).
\end{equation} 
\end{lemma}

\begin{proof} We may assume that convolution with $g$ is invertible on $\mathcal{S}_K(\M)$. Otherwise, 
the first supremum in \eqref{sphere-measure} is infinite. Let $G$ be the function on $I_\M$ that corresponds to $g$.
Since $\supp(g) \subset \mathcal{C}_{\delta}(\eta)$, we may write
\begin{equation*}
G\big(\cos(\gamma\,\ud(x,y))\big) 
= G\big(\cos(\gamma\,\ud(x,y))\big) \cdot\chi_{\mathcal{C}_{\delta}(y)}(x), \qquad x, y\in\M.
\end{equation*}
If $f^\ast\in \mathcal{S}_K(\M)$ is the unique function such that $f=f^\ast \ast g$, then by H\"older's inequality we have
\begin{multline} 
\int_{\M} |f|^p\,\ud\mu=
\int_{\M}\left|\int_{\M} f^\ast(y)G\big(\cos(\gamma\,\ud(x,y))\big) 
\chi_{\mathcal{C}_{\delta}(y)}(x)\,\ud\nu(y)\right|^p\,\ud\mu(x) \\
\leq \int_{\M} \int_{\M}|f^\ast(y)|^p\chi_{\mathcal{C}_{\delta}(y)}(x)\,\ud\nu(y)
\left(\int_{\M}|G\big(\cos(\gamma\,\ud(x,y))\big)|^q\,\ud\nu(y)\right)^{\frac{p}{q}}\,\ud\mu(x).
\label{fstar}
\end{multline}
From the invariance of the Haar measure~$\nu$, we infer that for every $x\in\M$,
\begin{equation*} 
\int_{\M}|G\big(\cos(\gamma\,\ud(x,y))\big)|^q\,\ud\nu(y)
 =\int_{\M}|G\big(\cos(\gamma\,\ud(\eta,z))\big)|^q\,\ud\nu(z)
 = \|g\|_q^q.
\end{equation*}
Substituting this into \eqref{fstar} and changing the order of integration, we obtain
\begin{eqnarray*}
\int_{\M} |f|^p\,\ud\mu
&\le& \|g\|_q^p \cdot \int_{\M} |f^\ast(y)|^p \;\mu(\mathcal{C}_{\delta}(y)) \; \ud\nu(y)\\
&\leq& \|g\|_q^p \cdot \|f^\ast\|^p_p \cdot \sup_{y\in\M}\mu(\mathcal{C}_{\delta}(y))\\
& =&\frac{\|f^\ast\|_p^p\|g\|_q^p}{\|f^\ast\ast g\|_p^p}\cdot\|f\|^p_p
\cdot \sup_{y\in\M}\mu(\mathcal{C}_{\delta}(y))\\
&\leq& \sup_{h\in \mathcal{S}_K(\M)}\frac{\|h\|_p^p\|g\|_q^p}{\|h\ast g\|_p^p} \cdot\|f\|^p_p \cdot
\sup_{y\in\M} \mu(\mathcal{C}_{\delta}(y))
\end{eqnarray*}
as claimed.
\end{proof}

 We denote the best constant in \eqref{sphere-measure} by
\begin{equation} \label{def-cpld}
T_p(K,\delta):=\inf_{g\in \mathcal{Z}^q_\delta(\M,\eta)} \;
\sup_{h\in  \mathcal{S}_K(\M)}\frac{\|h\|_p^p\|g\|_q^p}{\|h\ast g\|_p^p}.
\end{equation}
For $p=2$ we can explicitly calculate $T_2(K,\delta)$.

\begin{theorem}\label{thm:l2-const}
Let $\M$ be a two-point homogeneous space and $K\in\mathcal{N}$. If $t_{K,K}\leq \delta< 1$, then 
$g_{K,\delta}:=\chi_{\mathcal{C}_\delta(\eta)}\cdot P_K^\ab\big(\cos(\gamma\,\ud(\cdot,\eta)\big)$ 
is a minimizer for the extremal problem \eqref{def-cpld}, and the minimum is given by
\begin{equation} \label{C2-expl}
T_2(K,\delta) =  \left( \nu^\bot(\M^\bot)\int_{\delta}^1\frac{P_K^{(\alpha,\beta)}(t)^2}{P_K^\ab(1)^2}
\omega_{\alpha,\beta}(t)\,\ud t\right)^{-1}.
\end{equation}
\end{theorem}

\begin{remark}
According to \eqref{eq:niko}, it is enough to have
$$
1-2\frac{\alpha+1}{n(n+\alpha+\beta+1)}\leq\delta<1.
$$
\end{remark}

\begin{proof} First, we simplify the extremal problem \eqref{def-cpld}.  Let
$g\in\mathcal{Z}^2_\delta(\M,\eta)\setminus\{0\}$. Using the convolution
theorem (Theorem~\ref{conv-thm}) and Parseval's identity, we observe that
the quantity we intend to minimize is
\begin{eqnarray}
\sup_{\substack{h\in \mathcal{S}_K\\ h\neq0}}\frac{\|h\|_2^2\|g\|_2^2}{\|h\ast   g\|_2^2}
&=&\sup_{\substack{h\in \mathcal{S}_K\\ h\neq0}}\|g\|_2^2\|h\|_2^2\lp
\sum_{k = 0}^K\sum_{j = 1}^{d_k}\frac{|\widehat{h}(k,j)|^2\cdot|\widehat{g}(k,1)|^2}{d_k} \rp^{-1}\nonumber\\ 
&=&\max_{0\leq k\leq K}\dfrac{d_k\|g\|_2^2}{|\widehat{g}(k,1)|^2}, \label{extremal2}
\end{eqnarray} 
since Parseval's relation gives $\displaystyle \|h\|_2^2=\sum_{k = 0}^K\sum_{j = 1}^{d_k}|\widehat{h}(k,j)|^2$.

Next, as $g=g\cdot\chi_{\mathcal{C}_\delta}$, from Cauchy-Schwarz we get that
$$
|\widehat{g}(k,1)|=|\langle g\cdot\chi_{\mathcal{C}_\delta(\eta)},Y_{k,1}\rangle|
=|\langle g,\chi_{\mathcal{C}_\delta(\eta)}\cdot Y_{k,1}\rangle|
\leq\|g\|_2\|\chi_{\mathcal{C}_\delta(\eta)}\cdot Y_{k,1}\|_2.
$$
Further, equality occurs for $g$ a constant multiple of $g_{k,\delta}:=\chi_{\mathcal{C}_\delta(\eta)}\cdot Y_{k,1}$.
It follows that
$$
\inf_{g\in\mathcal{Z}^2_\delta(\M,\eta)}\sup_{\substack{h\in \mathcal{S}_K\\ h\neq0}}
\frac{\|h\|_2^2\|g\|_2^2}{\|h\ast   g\|_2^2}\\
=\max_{0\leq k\leq K}\dfrac{d_k}{\|g_{k,\delta}\|_2^2}
$$
On the other hand, as $\delta>t_{K,K}>t_{k,k}$ the largest zero of $P_k^\ab$, we have
$g_{k,\delta}\geq 0$. Therefore
\begin{eqnarray*}
\inf_{g\in\mathcal{Z}^2_\delta(\M,\eta)}\sup_{\substack{h\in \mathcal{S}_K\\ h\neq0}}
\frac{\|h\|_2^2\|g\|_2^2}{\|h\ast   g\|_2^2}
&=&\inf_{\substack{g\in\mathcal{Z}^2_\delta(\M,\eta)\\ g\geq 0}}\sup_{\substack{h\in \mathcal{S}_K\\ h\neq0}}
\frac{\|h\|_2^2\|g\|_2^2}{\|h\ast   g\|_2^2}\\
&=&\inf_{\substack{g\in\mathcal{Z}^2_\delta(\M,\eta)\\ g\geq 0}}\max_{0\leq k\leq K} \dfrac{d_k\|g\|_2^2}{|\widehat{g}(k,1)|^2}.
\end{eqnarray*}
But, if $g\in \mathcal{Z}^2_\delta(\M,\eta)\setminus\{0\}$ with $g\geq 0$, we may write
$g(x)=G\big(\cos(\gamma\,\ud(x,\eta))\big)$ with $G\geq 0$. Then
\begin{eqnarray*}
\frac{\widehat{g}(k,1)}{(d_k)^{1/2}} &=&\frac{1}{(d_k)^{1/2}}\int_\M g(x)\overline{Y_k^1(x)}\,\ud\nu(x)\\
&=&\frac{1}{P_k^\ab(1)}\int_\M G\big(\cos(\gamma\,\ud(x,\eta))\big)
P_k^\ab\big(\cos(\gamma\,\ud(x,\eta))\big)\,\ud\nu(x)\\
&=&\nu^\bot(\M^\bot)\int_\delta^1 G(t)\frac{P_k^\ab(t)}{P_k^\ab(1)}\omega_{\alpha,\beta}(t)\,\ud t.
\end{eqnarray*} 
Since $t_{K,K} \le \delta <1$, it follows from Lemma~\ref{ordering-PN} that
$$
\frac{\widehat{g}(k,1)}{(d_k)^{1/2} } 
\ge \nu^\bot(\M^\bot)\int_\delta^1 G(t) \frac{P_K^\ab (t)}{P_K^\ab(1)}  \omega_{\alpha,\beta}(t)\,\ud t
=\frac{ \widehat{g}(K,1)}{(d_K)^{1/2}}.
$$
Therefore
$$
\inf_{g\in\mathcal{Z}^2_\delta(\M,\eta)}\sup_{\substack{h\in \mathcal{S}_K\\ h\neq0}}
\frac{\|h\|_2^2\|g\|_2^2}{\|h\ast   g\|_2^2}
=\inf_{\substack{g\in\mathcal{Z}^2_\delta(\M,\eta)\\ g\geq 0}} \dfrac{d_K\|g\|_2^2}{|\widehat{g}(K,1)|^2}.
$$
and we have already seen that the function that realizes the minumum is $g_{K,\delta}$
for which the minimum is
$$
\inf_{g\in\mathcal{Z}^2_\delta(\M,\eta)}\sup_{\substack{h\in \mathcal{S}_K\\ h\neq0}}
\frac{\|h\|_2^2\|g\|_2^2}{\|h\ast   g\|_2^2}
=\dfrac{d_K}{\|g_{K,\delta}\|_2^2}.
$$
It remains to notice that
$$
\|g_{K,\delta}\|_2^2=\nu^\bot(\M^\bot)\int_\delta^1 P_K^\ab(t)^2\omega_{\alpha,\beta}(t) \ud t.
$$
From this, we deduce that
$$
\dfrac{d_K}{\|g_{K,\delta}\|_2^2}
 = P_K^\ab(1)^2\left( \nu^\bot(\M^\bot)\int_{\delta}^1P_K^\ab(t)^2\omega_{\alpha,\beta}(t) \ud t\right)^{-1}
$$
which is precisely $T_2(K,\delta)$.
\end{proof}

We are now able to state our  main theorem.

\begin{theorem}\label{thm:l2} 
Let $\M$ be a two-point homogeneous space, $\alpha$ be given by \eqref{def-alpha-beta}, $\mu$ be a $\sigma$-finite measure, $\Omega\subset\M$ be
measurable, and $t_{K,K}\leq\delta<1$. For $K \in\mathcal{N}$ and every
$f\in \mathcal{S}_K$, it holds
\begin{equation}\label{nonuniform-L2-mu}
\int_{\M}|f|^2\,\ud\mu
\leq \left(\nu^\bot(\M^\bot)\int_\delta^1\frac{P_K^\ab(t)^2}{P_K^\ab(1)^2}\omega_{\alpha,\beta}(t)\,\ud t\right)^{-1}
\cdot \|f\|_2^2 \cdot \sup_{y\in\S^2} \mu(\mathcal{C}_\delta(y)).
\end{equation}
Consequently, 
\begin{equation}\label{uniform-L2}
\lambda_{2}(\Omega,K) \leq A_{K}\cdot \rho(\Omega,K),
\end{equation}
where
\begin{equation}\label{eq-c2l}
A_{K}:=2^{\alpha+\beta+1}B_{(1-t_{K,K})/2}(\alpha+1,\beta+1)\left( \int_{t_{K,K}}^1
\frac{P_K^\ab(t)^2}{P_K^\ab(1)^2}\omega_{\alpha,\beta}(t)\,\ud t\right)^{-1},
\end{equation}
and $B_x(a,b)$ is the incomplete beta function.
\end{theorem}

\begin{proof} Combining Lemma~\ref{measure-prop} and
Theorem~\ref{thm:l2-const} gives \eqref{nonuniform-L2-mu}. Taking $\mu= \chi_\Omega\,\ud\nu$
in \eqref{nonuniform-L2-mu} and using \eqref{surf-meas-cap} and \eqref{def-max-nyqu}, we obtain 
for $f\in \mathcal{S}_K$ with $\|f\|_2=1$
\begin{eqnarray*} 
\int_{\Omega}|f|^2\,\ud\nu &\leq&T_p(K,t_{K,K}) \cdot \sup_{y\in\M} |\Omega\cap \mathcal{C}_{t_{K,K}}(y)|  \\
&\leq&  2^{\alpha+\beta+1}\nu^\bot(\M^\bot)B_{(1-{t_{K,K}})/2}(\alpha+1,\beta+1)\cdot T_p(K,t_{K,K})\\
&&\times
 \sup_{y\in\M} \frac{|\Omega\cap \mathcal{C}_{t_{K,K}}(y)| }{|\mathcal{C}_{t_{K,K}}(y)|} \\
& = &  \frac{2^{\alpha+\beta+1}B_{(1-{t_{K,K}})/2}(\alpha+1,\beta+1)}{ \int_{t_{K,K}}^1
\frac{P_K^\ab(t)^2}{P_K^\ab(1)^2}\omega_{\alpha,\beta}(t)\,\ud t }\cdot \rho(\Omega,K),
\end{eqnarray*}
which implies \eqref{uniform-L2}.
\end{proof}

\begin{lemma} \label{lemma-binfty}
If $\alpha$ and $\beta$ are given by \eqref{def-alpha-beta}, then
\begin{equation} \label{int-pl}
\lim_{K\rightarrow\infty} A_{K} 
= \frac{1}{\alpha+1}\left(\frac{j_{\alpha,1}}{2}\right)^{2\alpha}\frac{1}{J_{\alpha+1}(j_{\alpha,1})^{2}},
\end{equation}
where $J_\alpha$ is the Bessel function of the first kind, and $j_{\alpha,1}$ is
the smallest positive zero of the Bessel function $J_\alpha$.
\end{lemma}

\begin{proof} 
For simplicity, let us write $P_K=P_K^\ab$.

Recall  that by \eqref{eq:asymp-tKK} we have 
$t_{K,K}=1-\frac{j_{\alpha,1}^2}{2K^2}+\mathcal{O}(K^{-3})$. 
Let us first only consider the integral in \eqref{eq-c2l}  times the factor $(1-t_{K,K})^{-1-\alpha}$.
The squared Jacobi polynomial can be rewritten using Taylor's theorem with
the remainder in the Lagrange form 
\begin{multline*}
(1-t_{K,K})^{-1-\alpha} 
\int_{t_{K,K}}^1\frac{P_K(t)^2}{P_K(1)^2}(1-t)^\alpha (1+t)^\beta\,\ud t  \\
= \int_0^1 \frac{P_K\big(1-s(1-t_{K,K})\big)^2}{P_K(1)^2}s^\alpha (2-s(1-t_{K,K}))^\beta\,\ud s \\
=\int_0^1\frac{P_K\Big(1-\frac{ j_{\alpha,1}^2}{2K^2}\,s+h_Ks\Big)^2}
{P_K(1)^2}s^\alpha(2-s(1-t_{K,K}))^\beta\,\ud s\\
= \int_0^1 \frac{P_K\Big( 1-\frac{j_{\alpha,1}^2}{2K^2}\,s\Big)^2 +
2h_KsP_K(\xi_s)\frac{\ud}{\ud t}P_K(\xi_s)}{P_K(1)^2}s^\alpha(2-s(1-t_{K,K}))^\beta\,\ud s,
\end{multline*}
where $\xi_s\in \Big[1 -\frac{ j_{\alpha,1}^2}{2K^2}s,1 - \frac{j_{\alpha,1}^2}{2K^2}s+h_Ks\Big]$, and 
$h_K=\mathcal{O}(K^{-3})$ in view of \eqref{zeros-appr1}.

By \eqref{pn-max} and \eqref{pn-prime-max} we have that
$$
h_K\cdot\|P_K\|_\infty\cdot \|\frac{\ud}{\ud t}P_K\|_\infty\cdot P_K(1)^{-2}
= \mathcal{O}(K^{-3+\alpha+\alpha+2-2\alpha})=\mathcal{O}(K^{-1}).
$$
As $|P_K^\ab(t)|\leq P_K^\ab(1)$, it follows that we may apply the dominated convergence theorem. Therefore, it follows from the observation that $\displaystyle\dfrac{P_K(1)}{K^\alpha}\rightarrow 1$ and the
Mehler-Heine formula \eqref{mehler-heine1}  that the integral converges to
\begin{eqnarray*}
\frac{2^{\beta+2\alpha}}{j_{\alpha,1}^{2\alpha}}\int_0^1J_\alpha(j_{\alpha,1}\sqrt{s})^2\,\ud s
& =&\frac{2^{\beta+2\alpha}}{j_{\alpha,1}^{2(\alpha+1)}}2 \int_0^{j_{\alpha,1}}sJ_\alpha(s)^2\,\ud s\\ 
& =&\frac{2^{\beta+2\alpha}}{j_{\alpha,1}^{2(\alpha+1)}} s^2
\left(J_\alpha(s)^2-J_{\alpha-1}(s)J_{\alpha+1}(s)\right)\Big|_0^{j_{\alpha,1}}\\
& =&-\frac{2^{\beta+2\alpha}}{j_{\alpha,1}^{2\alpha}} 
J_{\alpha-1}(j_{\alpha,1})J_{\alpha+1}(j_{\alpha,1})\\
&=&\frac{2^{\beta+2\alpha}}{j_{\alpha,1}^{2\alpha}} J_{\alpha+1}(j_{\alpha,1})^2,
\end{eqnarray*}
where we have used that 
$$
-J_{\alpha-1}(z)J_{\alpha+1}(z)=J_{\alpha+1}(z)^2-\frac{2\alpha}{z}J_\alpha(z)J_{\alpha+1}(z);
$$
\emph{see e.g.} \cite[10.6.1]{NIST10}.
Note that the anti-derivative of the function $sJ_\alpha(s)^2$ is given in
\cite[5.54.2]{grary07}. 

It remains to study the convergence of the remaining factors defining $A_K$. In particular, we have by
\eqref{eq:incomplete-beta-asym} that
$$
2^{\alpha+\beta+1}\cdot B_{(1-{t_{K,K}})/2}(\alpha+1,\beta+1)\cdot(1-t_{K,K})^{-1-\alpha} 
= \frac{2^{\beta}}{\alpha+1}+\mathcal{O}(K^{-2(\alpha+2)}),
$$
which concludes the proof
\end{proof}

Using interpolation with the trivial inequality $\lambda_{\infty}(\Omega,K)\leq 1$ when $2<p<\infty$ 
and duality when $1<p<2$, we can extend \eqref{uniform-L2} to the case $1<p<\infty$.

\begin{theorem}\label{thm:Lp-est}
Let $\Omega\subset\M$ be measurable and $1<p<\infty$. For
$K\in\mathcal{N}$, it holds
\begin{equation*} 
\lambda_{p}(\Omega,K) =
\sup_{f\in\mathcal{S}_K\setminus\{0\}}\frac{\int_\Omega|f|^p\ud\sigma}{\int_{\S^2}|f|^p\,\ud\sigma}
\leq\big(A_K\cdot\rho(\Omega,K)\big)^{\min(p-1,1)} .
\end{equation*}
\end{theorem}
As the proof of this theorem works exactly as in \cite[Theorem~3.5]{spehry19}, we omit it here.

\section*{Acknowledgements}

\noindent M.S. was supported by the Austrian Science Fund (FWF) through  an Erwin-Schr{\"{o}}dinger Fellowship (J-4254).

\bibliographystyle{plain}

\begin{thebibliography}{}

\end{thebibliography}


\begin{thebibliography}{10}

\bibitem{abspe17-sampta}
L.~D. Abreu and M.~Speckbacher.
\newblock A planar large sieve and sparsity of time-frequency representations.
\newblock In {\em Proceedings of SampTA}, 2017.

\bibitem{abspe18-sieve}
L.~D. Abreu and M.~Speckbacher.
\newblock Donoho-{L}ogan large sieve principles for modulation and polyanalytic
  {F}ock spaces.
\newblock {\em arXiv:1808.02258}, 2018.

\bibitem{alsasne99}
A.~Albertella, F.~Sanso, and N.~Sneeuw.
\newblock Band-limited functions on a bounded spherical domain: the {S}lepian
  problem on the sphere.
\newblock {\em J. Geod.}, 73(9):436--447, 1999.

\bibitem{ca29}
E.~Cartan.
\newblock Sur la dermination d'un systeme orthogonal complet dans un espace de
  {R}iemann symetrique clos.
\newblock {\em Rendiconti Circ. mat. di Palermo}, 52(217--252), 1929.

\bibitem{dolo92}
D.~L. Donoho and B.~F. Logan.
\newblock Signal recovery and the large sieve.
\newblock {\em SIAM J. Appl. Math.}, 52(2):577--591, 1992.

\bibitem{freeden1996}
W.~Freeden and U.~Windheuser.
\newblock Spherical wavelet transform and its discretization.
\newblock {\em Adv. Comp. Math.}, 5(1):51--94, 1996.

\bibitem{ga67}
R.~Gangolli.
\newblock Positive definite kernels on homogeneous spaces and certain
  stochastic processes related to {L}{\'e}vy's {B}rownian motion of several
  parameters.
\newblock {\em Ann. Inst. H. Poincar{\'e}}, 3(2):121--225, 1967.

\bibitem{grary07}
I.~S. Gradshteyn and I.~M. Ryzhik.
\newblock {\em Table of {I}ntegrals, {S}eries and {P}roducts}.
\newblock Academic Press, 7th edition, 2007.

\bibitem{hel62}
S.~Helgason.
\newblock {\em Differential {G}eometry and {S}ymmetric {S}pacecs}.
\newblock Academic Press, New York, 1962.

\bibitem{hel65}
S.~Helgason.
\newblock The {R}adon transform on {E}uclidean spaces, compact two-point
  homogeneous spaces and {G}rassmannn manifolds.
\newblock {\em Acta Math.}, 1965.

\bibitem{kennedy2011}
R.~A. Kennedy, T.~A. Lamahewa, and L.~Wei.
\newblock On azimuthally symmetric 2-sphere convolution.
\newblock {\em Digital Signal Processing}, 21(5):660--666, 2011.

\bibitem{ko73}
T.~Koornwinder.
\newblock The addition formula for {J}acobi polynomials and spherical
  harmonics.
\newblock {\em {SIAM} J. Appl. Math.}, 25(2):236--246, 1973.

\bibitem{lan61}
H.~J. Landau.
\newblock Prolate spheroidal wave functions, {F}ourier analysis and uncertainty
  {II}.
\newblock {\em Bell Syst. Tech. J.}, 40:65--84, 1961.

\bibitem{lapo62}
H.~J. Landau and H.~O. Pollak.
\newblock Prolate spheroidal wave functions, {F}ourier analysis and uncertainty
  {III}: {T}he dimension of the space of essentially time‐and band‐limited
  signals.
\newblock {\em Bell Syst. Tech. J.}, 41(4):1295–1336, 1962.

\bibitem{mont78}
H.~L. Montgomery.
\newblock The analytic principle of the large sieve.
\newblock {\em Bull. Amer. Math. Soc.}, 84(4), 1978.

\bibitem{narcowich1996}
F.~J. Narcowich and J.~D. Ward.
\newblock Nonstationary wavelets on the $m$-sphere for scattered data.
\newblock {\em Appl. Comp. Harmonic Anal.}, 3(4):324--336, 1996.

\bibitem{Nikolov1}
G.~Nikolov.
\newblock New bounds for the extreme zeros of {J}acobi polynomials.
\newblock {\em Proc. Amer. Math. Soc.}, 11:1541–1550, 2019.

\bibitem{Nikolov2}
G.~Nikolov.
\newblock On the extreme zeros of {J}acobi polynomials.
\newblock ar{X}iv:2002.02633, 2020.

\bibitem{NIST10}
F.W. Olver, D.~W. Lozier, R.~F. Boisvert, and C.~W. Clark.
\newblock {\em {NIST} {H}andbook of {M}athematical {F}unctions}.
\newblock Cambridge University Press, 1st edition, 2010.

\bibitem{orpri13}
J.~Ortega-Cerd{\'a} and B.~Pridhnani.
\newblock Carleson measures and {L}ogvinenko-{S}ereda sets on compact
  manifolds.
\newblock {\em Forum Math.}, 25(1):151--172, 2013.

\bibitem{sidawi06}
F.~J. Simons, F.~A. Dahlen, and M.~A. Wieczorek.
\newblock Spatiospectral concentration on the sphere.
\newblock {\em SIAM Review}, 48(3):504--536, 2006.

\bibitem{slepo61}
D.~Slepian and H.~O. Pollak.
\newblock Prolate spheroidal wave functions, {F}ourier analysis and uncertainty
  {I}.
\newblock {\em Bell Syst. Tech. J.}, 40(1):43--63, 1961.

\bibitem{spehry19}
M.~Speckbacher and T.~Hrycak.
\newblock Concentration estimates for band-limited spherical harmonics
  expansions via the large sieve principle.
\newblock {\em J. Fourier Anal. Appl.}, 2020.

\bibitem{wa52}
H.~C. Wang.
\newblock Two-point homogeneous spaces.
\newblock {\em Ann. Math.}, 55:177--191, 1952.

\bibitem{Zelditch}
S. {Zelditch}.
\newblock {\em {Eigenfunctions of the Laplacian on a Riemannian manifold.}},
  volume 125.
\newblock Providence, RI: American Mathematical Society (AMS), published for
  the Conference Board of the Mathematical Sciences (CBMS), 2017.

\end{thebibliography}

\end{document}